\documentclass[a4paper,twoside]{article}

\def\shorttitle{Strong Maximum Principle for Fractional Diffusion Equations}
\def\shortauthor{Y. Liu}

\usepackage{amsfonts}
\usepackage{amsmath}
\usepackage{amssymb}
\usepackage{amscd}
\usepackage{amsbsy}
\usepackage{amsthm}
\usepackage{bm}
\usepackage{empheq}
\usepackage{graphicx}
\usepackage{psfrag}
\usepackage{color}
\usepackage{cite}
\usepackage{multicol}

\allowdisplaybreaks[4]

\newfont{\myfnt}{cmssi10 scaled 1440}
\numberwithin{equation}{section}
\catcode`@=11
\def\ps@nk{\def\@oddhead{\vbox{\hbox to \hsize{\pic \footnotesize \it \shorttitle
\hfill \rm \thepage} \vspace{1mm} \vspace*{-2mm}}}
\def\@evenhead{\vbox{\hbox to \hsize{\pic \footnotesize \rm \thepage \hfill \it \shortauthor}
\vspace{1mm} \vspace*{-2mm}}}
\def\@oddfoot{} \def\@evenfoot{}}
\def\ps@first{\def\@oddhead{\vbox{\hbox to \hsize{\pic \footnotesize
} \break}}
\def\@oddfoot{} \def\@evenfoot{}}
\newtheoremstyle{thmstyle}% name
  {6pt}%      Space above
  {6pt}%      Space below
  {\it}%         Body font
  {}%         Indent amount (empty = no indent, \parindent = para indent)
  {\bf}% Thm head font
  {}%        Punctuation after thm head
  {.5em}%     Space after thm head: " " = normal interword space;
  {}%         Thm head spec (can be left empty, meaning `normal')
\newtheoremstyle{remstyle}% name
  {6pt}%      Space above
  {6pt}%      Space below
  {\rm}%         Body font
  {}%         Indent amount (empty = no indent, \parindent = para indent)
  {\bf}% Thm head font
  {}%        Punctuation after thm head
  {.5em}%     Space after thm head: " " = normal interword space;
  {}%         Thm head spec (can be left empty, meaning `normal')
\def\Section#1{\Sec{\large #1} \setcounter{equation}{0} \vskip -6mm \indent}
\def\Sec{\@Startsection{section}{1}{\z@}
                                   {-3.5ex \@plus -1ex \@minus -.2ex}%
                                   {2.3ex \@plus.2ex}%
                                   {\normalfont\large\bfseries\boldmath}}
\def\@Startsection#1#2#3#4#5#6{%
  \if@noskipsec \leavevmode \fi
  \par
  \@tempskipa #4\relax
  \@afterindenttrue
  \ifdim \@tempskipa <\z@
    \@tempskipa -\@tempskipa \@afterindentfalse
  \fi
  \if@nobreak
    \everypar{}%
  \else
    \addpenalty\@secpenalty\addvspace\@tempskipa
  \fi
  \@ifstar
    {\@ssect{#3}{#4}{#5}{#6}}%
    {\@dblarg{\@Sect{#1}{#2}{#3}{#4}{#5}{#6}}}}
\def\@Sect#1#2#3#4#5#6[#7]#8{%
  \ifnum #2>\c@secnumdepth
    \let\@svsec\@empty
  \else
    \refstepcounter{#1}%
    \protected@edef\@svsec{\@seccntformat{#1}\relax}%
  \fi
  \@tempskipa #5\relax
  \ifdim \@tempskipa>\z@
    \begingroup
      #6{%
          \@hangfrom{\hskip #3\relax\@svsec \hskip -2.5mm}%
          \interlinepenalty \@M #8\@@par}
    \endgroup
    \csname #1mark\endcsname{#7}%
    \addcontentsline{toc}{#1}{%
      \ifnum #2>\c@secnumdepth \else
        \protect\numberline{\csname the#1\endcsname}%
      \fi
      #7}%
  \else
    \def\@svsechd{%
      #6{\hskip #3\relax
      \@svsec #8}%
      \csname #1mark\endcsname{#7}%
      \addcontentsline{toc}{#1}{%
        \ifnum #2>\c@secnumdepth \else
          \protect\numberline{\csname the#1\endcsname}%
        \fi
        #7}}%
  \fi
  \@xsect{#5}}
\renewenvironment{abstract}{%
        \small
        \quotation
         \noindent {\bfseries \abstractname } }%
      {\if@twocolumn\else\endquotation\fi}

\def\Subsec{\@StartSubsection{subsection}{2}{\z@}%
                                     {-3.25ex\@plus -1ex \@minus -.2ex}%
                                     {1.5ex \@plus .2ex}%
                                     {\normalfont\normalsize\bfseries\boldmath}}
\def\@StartSubsection#1#2#3#4#5#6{%
  \if@noskipsec \leavevmode \fi
  \par
  \@tempskipa #4\relax
  \@afterindenttrue
  \ifdim \@tempskipa <\z@
    \@tempskipa -\@tempskipa \@afterindentfalse
  \fi
  \if@nobreak
    \everypar{}%
  \else
    \addpenalty\@secpenalty\addvspace\@tempskipa
  \fi
  \@ifstar
    {\@ssect{#3}{#4}{#5}{#6}}%
    {\@dblarg{\@SubSect{#1}{#2}{#3}{#4}{#5}{#6}}}}
\def\@SubSect#1#2#3#4#5#6[#7]#8{%
  \ifnum #2>\c@secnumdepth
    \let\@svsec\@empty
  \else
    \refstepcounter{#1}%
    \protected@edef\@svsec{\@seccntformat{#1}\relax}%
  \fi
  \@tempskipa #5\relax
  \ifdim \@tempskipa>\z@
    \begingroup
      #6{%
          \@hangfrom{\hskip #3\relax\@svsec\hskip -1.5mm}%
          \interlinepenalty \@M #8\@@par}
    \endgroup
    \csname #1mark\endcsname{#7}%
    \addcontentsline{toc}{#1}{%
      \ifnum #2>\c@secnumdepth \else
        \protect\numberline{\csname the#1\endcsname}%
      \fi
      #7}%
  \else
    \def\@svsechd{%
      #6{\hskip #3\relax
      \@svsec #8}%
      \csname #1mark\endcsname{#7}%
      \addcontentsline{toc}{#1}{%
        \ifnum #2>\c@secnumdepth \else
          \protect\numberline{\csname the#1\endcsname}%
        \fi
        #7}}%
  \fi
  \@xsect{#5}}
\def\list#1#2{\ifnum \@listdepth >5\relax \@toodeep \else \global
\advance \@listdepth\@ne \fi \rightmargin \z@ \listparindent\z@
\itemindent\z@ \csname @list\romannumeral\the\@listdepth\endcsname
\def\@itemlabel{#1}\let\makelabel\@mklab \@nmbrlistfalse #2\relax
\@trivlist \parskip 0pt \parindent\listparindent \advance \linewidth
-\rightmargin \advance\linewidth -\leftmargin \advance\@totalleftmargin
\leftmargin \parshape \@ne \@totalleftmargin \linewidth \ignorespaces}
\renewcommand{\@makecaption}[2]{\begin{center}#1. #2\end{center}}
\catcode`@=12 
\theoremstyle{thmstyle}
\newtheorem{thm}{\indent Theorem}[section]
\newtheorem{lem}[thm]{\indent Lemma}

\newtheorem{coro}[thm]{\indent Corollary}

\theoremstyle{remstyle}

\newsavebox{\mygraphic}
\def\pic{\begin{picture}(0,0) \put(-210,-1250){\usebox{\mygraphic}} \end{picture}}
\newfont{\HUGEbf}{cmbx10 scaled 3500}
\definecolor{gray}{rgb}{0.9,0.9,0.9}
\def\thebibliography#1{\section*{\bf \large References}
\list{[\arabic{enumi}]} {\settowidth \labelwidth{[#1]} \leftmargin
\labelwidth \advance \leftmargin \labelsep \usecounter{enumi}}
\def\newblock{\hskip .11em plus .33em minus .07em} \footnotesize \sloppy \clubpenalty
4000 \widowpenalty 4000 \sfcode`\.=1000 \relax}

\def\BC{\mathbb C}
\def\BN{\mathbb N}
\def\BR{\mathbb R}
\def\cA{\mathcal A}
\def\cD{\mathcal D}
\def\cE{\mathcal E}

\def\rd{\mathrm d}

\def\supp{\mathrm{supp}}

\def\Ga{\Gamma}

\def\Om{\Omega}
\def\al{\alpha}
\def\be{\beta}
\def\ga{\gamma}
\def\de{\delta}

\def\ve{\varepsilon}

\def\la{\lambda}

\def\vp{\varphi}
\def\om{\omega}
\def\f{\frac}

\def\ov{\overline}
\def\pa{\partial}

\def\wt{\widetilde}

\theoremstyle{definition}

\numberwithin{equation}{section}

\title{\Large\bf\boldmath Strong Maximum Principle for Multi-Term Time-\\
Fractional Diffusion Equations and its Application\\
to an Inverse Source Problem$^*$}

\author{\large Yikan LIU$^\dag$}

\date{}

\begin{document}

\maketitle

\thispagestyle{first}
\renewcommand{\thefootnote}{\fnsymbol{footnote}}

\footnotetext{\hspace*{-5mm} \begin{tabular}{@{}r@{}p{14cm}@{}} &
Manuscript last updated: \today.\\
$^\dag$ & Graduate School of Mathematical Sciences, The University of Tokyo, 3-8-1 Komaba, Meguro-ku, Tokyo 153-8914, Japan. E-mail: ykliu@ms.u-tokyo.ac.jp\\
$^*$ & This work is partly supported by a Grant-in-Aid for Scientific Research (S) 15H05740 and the A3 Foresight Program ``Modeling and Computation of Applied Inverse Problems'', Japan Society of the Promotion of Science.
\end{tabular}}

\renewcommand{\thefootnote}{\arabic{footnote}}

\begin{abstract}
In this paper, we establish a strong maximum principle for fractional diffusion equations with multiple Caputo derivatives in time, and investigate a related inverse problem of practical importance. Exploiting the solution properties and the involved multinomial Mittag-Leffler functions, we improve the weak maximum principle for the multi-term time-fractional diffusion equation to a stronger one, which is parallel to that for its single-term counterpart as expected. As a direct application, we prove the uniqueness for determining the temporal component of the source term with the help of the fractional Duhamel's principle for the multi-term case.

\vskip 4.5mm

\noindent\begin{tabular}{@{}l@{ }p{10cm}} {\bf Keywords } & Fractional diffusion equation, Strong maximum principle,\\
& Multinomial Mittag-Leffler function, Inverse source problem
\end{tabular}

\vskip 4.5mm


\end{abstract}

\baselineskip 14pt

\setlength{\parindent}{1.5em}

\setcounter{section}{0}

\Section{Introduction and main results}\label{sec-intro}

Let $T>0$ and $\Om\subset\BR^d$ ($d=1,2,3$) be an open bounded domain with a smooth boundary (for example, of $C^\infty$ class). Fix a positive integer $m$ and let $\al_j,q_j$ ($j=1,\ldots,m$) be positive constants such that $1>\al_1>\cdots>\al_m>0$ and $q_1=1$ without loss of generality. Consider the following initial-boundary value problem for a time-fractional diffusion equation
\begin{equation}\label{eq-ibvp-u}
\left\{\!\begin{alignedat}{2}
& \sum_{j=1}^mq_j\pa_t^{\al_j}u(x,t)+\cA u(x,t)=F(x,t) & \quad & (x\in\Om,\ 0<t\le T),\\
& u(x,0)=a(x) & \quad & (x\in\Om),\\
& u(x,t)=0 & \quad & (x\in\pa\Om,\ 0<t\le T),
\end{alignedat}\right.
\end{equation}
where $\pa_t^{\al_j}$ denotes the Caputo derivative defined by
\[\pa_t^{\al_j}f(t):=\f1{\Ga(1-\al_j)}\int_0^t\f{f'(s)}{(t-s)^{\al_j}}\,\rd s,\]
and $\Ga(\,\cdot\,)$ denotes the Gamma function. Here $\cA$ is an elliptic operator defined for $f\in\cD(\cA):=H^2(\Om)\cap H^1_0(\Om)$ as
\begin{equation}\label{eq-def-A}
\cA f(x)=-\sum_{i,j=1}^d\pa_j(a_{ij}(x)\pa_if(x))+c(x)f(x)\quad(x\in\Om),
\end{equation}
where $a_{ij}=a_{ji}\in C^1(\ov\Om)$ ($1\le i,j\le d$) and $c\in C(\ov\Om)$. Moreover, it is assumed that $c\ge0$ in $\ov\Om$ and there exists a constant $\de>0$ such that
\[\sum_{i,j=1}^da_{ij}(x)\xi_i\xi_j\ge\de\sum_{i=1}^d\xi_i^2\quad(\forall\,x\in\ov\Om\,,\ \forall\,(\xi_1,\ldots,\xi_d)\in\BR^d).\]
The assumptions on the initial value $a$ and the source term $F$ will be specified later. For later convenience, we abbreviate $\bm\al:=(\al_1,\ldots,\al_m)$ and $\bm q:=(q_1,\ldots,q_m)$.

The governing equation in problem \eqref{eq-ibvp-u} is called a single-term time-fractional diffusion equation for $m=1$ and is called a multi-term one for $m\ge2$. Especially, the single-term case of \eqref{eq-ibvp-u} has been utilized extensively as a model for the anomalous diffusion phenomena in heterogeneous media (see e.g. \cite{AG92,HH98,MK00} and the references therein). As a natural generalization, the multi-term case is expected to improve the modeling accuracy for the anomalous diffusion, which has also drawn increasing attentions of mathematicians. In Luchko \cite{L10} and \cite{L11}, explicit solutions to \eqref{eq-ibvp-u} were given for $m=1$ and $m\ge2$ respectively. Sakamoto and Yamamoto \cite{SY11} established the fundamental well-posedness theory for \eqref{eq-ibvp-u} with $m=1$, which was parallelly extended to the case of $m\ge2$ in Li, Liu and Yamamoto \cite{LLY15}. Regarding numerical treatments, we refer to \cite{LZATB07,JLZ13} for the single-term case and \cite{JLLZ15} for the multi-term case. Meanwhile, although mainly restricted to the single-term case, \eqref{eq-ibvp-u} has also gained population among the inverse problem researchers in the last few years; recent literatures include \cite{JR15,LY15,LRYZ13}. Comparing the solutions of fractional diffusion equations with that of classical parabolic equations (i.e., $m=\al=1$ in \eqref{eq-ibvp-u}) especially on the regularity and asymptotic behavior, one can easily conclude from the existing works that they differ considerably from each other in the senses of the smoothing effect in space and the decay rate in time.

In contrast with the above mentioned aspects, it reveals that fractional diffusion equations possess certain similarity to their integer prototypes on the maximum principle. In retrospect, Luchko \cite{L09} established a weak maximum principle for \eqref{eq-ibvp-u} with $m=1$, which was generalized to the case of $m\ge2$ in \cite{L11}. As further extensions, a similar weak maximum principle for the multi-term time-space Riesz-Caputo fractional differential equations was proved in \cite{YLAT14}, and Al-Refai and Luchko \cite{AL15} obtained a strong maximum principle for multi-term time-fractional diffusion equations with Riemann-Liouville derivatives. Very recently, Liu, Rundell and Yamamoto \cite{LRY15} proved a strong maximum principle for the single-term case of \eqref{eq-ibvp-u} and shown the uniqueness for a related inverse source problem as a direct application.

As a generalization of \cite{LRY15}, this paper aims at the establishment of parallel results for the fractional diffusion equations with multiple Caputo derivatives in time. Taking advantage of the weak maximum principle obtained in \cite{L11} (see Lemma \ref{lem-wmp}), first we can validate the following strong maximum principle for the initial-boundary value problem \eqref{eq-ibvp-u}.

\begin{thm}\label{thm-smp}
Let $a\in L^2(\Om)$ satisfy $a\ge0$ and $a\not\equiv0$, $F\in L^\infty(0,\infty;L^2(\Om))$, and $u$ be the solution to \eqref{eq-ibvp-u}.

{\rm(a)} If $F=0$, then for any $x\in\Om$, the set $\cE_x:=\{t>0;\,u(x,t)\le0\}$ is a finite set. Further, we have $u>0$ a.e.\! in $\Om\times(0,\infty)$.

{\rm(b)} If $F\ge0$, then $u>0$ a.e.\! in $\Om\times(0,\infty)$.
\end{thm}

The above conclusion is in principle consistent with its single-term counterpart (see \cite[Theorem 1.1]{LRY15}). Meanwhile, in view of the established weak maximum principle, Theorem \ref{thm-smp}(b) reduces to an immediate corollary to Theorem \ref{thm-smp}(a). On the other hand, it follows from the Sobolev embedding and Lemma \ref{lem-lly}(a) that $u$ allows a pointwise definition in Theorem \ref{thm-smp}(a) and thus the set $\cE_x$ is well-defined. Parallelly to the single-term case, $\cE_x$ is actually the set of zero points of $u(x,t)$ as a function of $t>0$, which contains at most finite elements according to Theorem \ref{thm-smp}(a). Unfortunately, the desired strict positivity, i.e., $\cE_x=\emptyset$ ($\forall\,x\in\Om$), still remains open due to technical difficulties, although the following result can be easily demonstrated.

\begin{coro}\label{coro-smp}
Let $a\in L^2(\Om)$ satisfy $a>0$ a.e.\! in $\Om$, $F\in L^\infty(0,\infty;L^2(\Om))$ be non-negative, and $u$ be the solution to \eqref{eq-ibvp-u}. Then $u>0$ a.e.\! in $\Om\times(0,\infty)$.
\end{coro}

Similarly to \cite{LRY15}, we can investigate a related inverse source problem and give a uniqueness result by applying Theorem \ref{thm-smp}.

\begin{thm}\label{thm-ISP}
Fix $x_0\in\Om$ and $T>0$ arbitrarily and let $u$ be the solution to \eqref{eq-ibvp-u} with $a=0$ and $F(x,t)=\rho(t)\,g(x)$. Assume that $\rho\in C^1[0,T]$, $g\in\cD(\cA^\ve)$ with some $\ve>0\ ($see Section $\ref{sec-pre}$ for the definition of $\cD(\cA^\ve))$, $g\ge0$ and $g\not\equiv0$. Then $u(x_0,t)=0\ (0\le t\le T)$ implies $\rho(t)=0\ (0\le t\le T)$.
\end{thm}

The above theorem gives an affirmative answer to the inverse problem on the reconstruction of the temporal component $\rho$ in the inhomogeneous term $F(x,t)=\rho(t)\,g(x)$ in \eqref{eq-ibvp-u} by the single point observation. As was explained in \cite{LRY15}, such an inverse source problem simulates the situation of determining the time evolution pattern of the anomalous diffusion caused by the space-dependent source $g$ of contaminants. On the same problem, a double-sided stability estimates was already obtained in \cite{SY11} under the assumption that $x_0\in\supp\,g$. However, from a practical viewpoint, it is preferable that the monitoring point $x_0$ locates far away from the support of $g$. In this sense, Theorem \ref{thm-ISP} seems more realistic because the choice of $x_0$ is arbitrary.

Nevertheless, we emphasize that in the case of $x_0\notin\supp\,g$, the non-negative and non-vanishing assumptions of $g$ in Theorem \ref{thm-ISP} is essential for the uniqueness. Similarly to that constructed in \cite{LRY15}, we can also take advantage of the explicit solution to give a simple counterexample for the multi-term case where the data on $\{x_0\}\times[0,T]$ does not guarantee the uniqueness. Here we omit the details.

The rest of this paper is organized as follows. In Section \ref{sec-pre} we recall the necessary ingredients concerning problem \eqref{eq-ibvp-u} from existing works and show a key lemma for proving the main results. Sections \ref{sec-proof-smp} is devoted to the proofs of Theorem \ref{thm-smp} and Corollary \ref{coro-smp}, and the proof of Theorem \ref{thm-ISP} is given in Section \ref{sec-proof-ISP}.

\Section{Preliminaries}\label{sec-pre}

We start from introducing the general notations for self-completeness. Let $L^2(\Om)$ be the usual $L^2$-space equipped with the inner product $(\,\cdot\,,\,\cdot\,)$ and $H^1_0(\Om)$, $H^2(\Om)$ denote the Sobolev spaces. Let $\{(\la_n,\vp_n)\}_{n=1}^\infty$ be the eigensystem of the symmetric uniformly elliptic operator $\cA$ in \eqref{eq-ibvp-u} such that $0<\la_1<\la_2\le\cdots$, $\la_n\to\infty$ as $n\to\infty$ and $\{\vp_n\}\subset H^2(\Om)\cap H^1_0(\Om)$ forms a complete orthonormal basis of $L^2(\Om)$. Then we can define the fractional power $\cA^\ga$ for $\ga\ge0$ as
\[\cD(\cA^\ga)=\left\{f\in L^2(\Om);\,\sum_{n=1}^\infty|\la_n^\ga\,(f,\vp_n)|^2<\infty\right\},\quad\cA^\ga f:=\sum_{n=1}^\infty\la_n^\ga\,(f,\vp_n)\,\vp_n,\]
and $\cD(\cA^\ga)$ is a Hilbert space with the norm
\[\|f\|_{\cD(\cA^\ga)}=\left(\sum_{n=1}^\infty|\la_n^\ga\,(f,\vp_n)|^2\right)^{1/2}.\]

For $1\le p\le\infty$, $0<T\le\infty$ and a Banach space $X$, we say that $f\in L^p(0,T;X)$ provided
\[\|f\|_{L^p(0,T;X)}:=\left\{\!\begin{alignedat}{2}
& \left(\int_0^T\|f(\,\cdot\,,t)\|_X^p\,\rd t\right)^{1/p} & \quad & \mbox{if }1\le p<\infty\\
& \mathop{\mathrm{ess}\sup}_{0<t<T}\|f(\,\cdot\,,t)\|_X & \quad & \mbox{if }p=\infty
\end{alignedat}\right\}<\infty.\]
Similarly, for $0\le t_0<T<\infty$, we say that $f\in C([t_0,T];X)$ provided
\[\|f\|_{C([t_0,T];X)}:=\max_{t_0\le t\le T}\|f(\,\cdot\,,t)\|_X<\infty.\]
In addition, we define
\[C((0,T];X):=\bigcap_{0<t_0<T}C([t_0,T];X),\quad C([0,\infty);X):=\bigcap_{T>0}C([0,T];X).\]

To represent the explicit solution of \eqref{eq-ibvp-u}, we first recall the multinomial Mittag-Leffler function defined as follows. For $\be_0>0$, $\bm\be=(\be_1,\ldots,\be_m)\in(0,1)^m$ and $\bm z=(z_1,\ldots,z_m)\in\BC^m$, define (see Luchko and Gorenflo \cite{LG99})
\begin{equation}\label{eq-def-ML}
E_{\bm\be,\be_0}(\bm z):=\sum_{k=0}^\infty\sum_{k_1+\cdots+k_m=k}\f{(k;k_1,\ldots,k_m)\prod_{j=1}^mz_j^{k_j}}{\Ga(\be_0+\bm k\cdot\bm\be)},
\end{equation}
where $\bm k:=(k_1,\ldots,k_m)$, $\bm k\cdot\bm\be:=\sum_{j=1}^mk_j\be_j$ and $(k;k_1,\ldots,k_m)$ denotes the multinomial coefficient
\[(k;k_1,\ldots,k_m):=\f{k!}{k_1!\cdots k_m!}\quad\mbox{with }k=\sum_{j=1}^mk_j,\ k_j\ge0\ (1\le j\le m).\]
Obviously, the above definition is a natural generalization of the traditional one with $m=1$ (see, e.g., Podlubny \cite{P99}). For later use, we state several important properties of the multinomial Mittag-Leffler function.

\begin{lem}\label{lem-ML}
{\rm(a)} Let $\be_0>0$, $\bm\be\in(0,1)^m$ and $\bm z\in\BC^m$ be fixed. Then
\[E_{\bm\be,\be_0}(\bm z)=\f1{\Ga(\be_0)}+\sum_{j=1}^mz_j\,E_{\bm\be,\be_0+\be_j}(\bm z).\]

{\rm(b)} Let $\be_0>0$ and $1>\al_1>\cdots>\al_m>0$ be given. Let $\bm z\in\BC^m$ satisfy $\mu\le|\arg\,z_1|\le\pi$ and $-K\le z_j<0\ (j=2,\ldots,m)$ for some fixed $\mu\in(\al_1\pi/2,\al_1\pi)$ and $K>0$. Then there exists a constant $C>0$ depending only on $\mu$, $K$, $\al_j\ (j=1,\ldots,m)$ and $\be_0$ such that
\[E_{(\al_1,\al_1-\al_2,\ldots,\al_1-\al_m),\be_0}(\bm z)\le\f C{1+|z_1|}.\]
\end{lem}

We observe that the above properties are almost identical to \cite[Lemmata 3.1--3.2]{LLY15} except that the restriction $\be_0\in(0,2)$ is relaxed to $\be_0>0$. In fact, one can show Lemma \ref{lem-ML} with $\be_0>0$ by the same arguments as that in \cite{LLY15}, and here we omit the details.

Now we concentrate more on Lemma \ref{lem-ML}(a), which provides a link between multinomial Mittag-Leffler functions with different coefficients. Since $1/\Ga(\be_0)$ is the first term in the series by which $E_{\bm\be,\be_0}(\bm z)$ is defined, we see
\[\sum_{j=1}^mz_j\,E_{\bm\be,\be_0+\be_j}(\bm z)=\sum_{k=1}^\infty\sum_{k_1+\cdots+k_m=k}\f{(k;k_1,\ldots,k_m)\prod_{j=1}^mz_j^{k_j}}{\Ga(\be_0+\bm k\cdot\bm\be)},\]
that is, the remainder after the first term in \eqref{eq-def-ML} can be written by a linear combination of multinomial Mittag-Leffler functions. Actually, the following lemma reveals that the same result holds for any remainder in \eqref{eq-def-ML}.

\begin{lem}\label{lem-remainder}
Denote by $R_{\bm\be,\be_0}^{\,\ell}(\bm z)$ the remainder after the $\ell$th term in \eqref{eq-def-ML}, i.e.,
\[R_{\bm\be,\be_0}^{\,\ell}(\bm z):=\sum_{k=\ell}^\infty\sum_{k_1+\cdots+k_m=k}\f{(k;k_1,\ldots,k_m)\prod_{j=1}^mz_j^{k_j}}{\Ga(\be_0+\bm k\cdot\bm\be)}\quad(\ell=1,2,\ldots).\]
Then there holds
\begin{equation}\label{eq-remainder}
R_{\bm\be,\be_0}^{\,\ell}(\bm z)=\sum_{\ell_1+\cdots+\ell_m=\ell}(\ell;\ell_1,\ldots,\ell_m)\prod_{j=1}^mz_j^{\ell_j}E_{\bm\be,\be_0+\bm\ell\cdot\bm\be}(\bm z)\quad(\ell=1,2,\ldots),
\end{equation}
where $\bm\ell:=(\ell_1,\ldots,\ell_m)$.
\end{lem}

\begin{proof}
It is natural to show by induction with respect to $\ell=1,2,\ldots$ because the case of $\ell=1$ is verified in Lemma \ref{lem-ML}(a).

Now suppose that \eqref{eq-remainder} is valid for some $\ell$. Then it suffices to verify that the summation after the $(\ell+1)$th term in \eqref{eq-def-ML} preserves the form of \eqref{eq-remainder}, that is,
\begin{align*}
R_{\bm\be,\be_0}^{\,\ell}(\bm z) & =\sum_{\ell_1+\cdots+\ell_m=\ell}\f{(\ell;\ell_1,\ldots,\ell_m)\prod_{j=1}^mz_j^{\ell_j}}{\Ga(\be_0+\bm\ell\cdot\bm\be)}\\
& \quad\,+\sum_{\ell_1+\cdots+\ell_m=\ell+1}(\ell+1;\ell_1,\ldots,\ell_m)\prod_{j=1}^mz_j^{\ell_j}E_{\bm\be,\be_0+\bm\ell\cdot\bm\be}(\bm z).
\end{align*}
To this end, we apply Lemma \ref{lem-ML}(a) to \eqref{eq-remainder} to deduce
\begin{align}
R_{\bm\be,\be_0}^{\,\ell}(\bm z) & =\sum_{\ell_1+\cdots+\ell_m=\ell}(\ell;\ell_1,\ldots,\ell_m)\prod_{j=1}^mz_j^{\ell_j}E_{\bm\be,\be_0+\bm\ell\cdot\bm\be}(\bm z)\nonumber\\
& =\sum_{\ell_1+\cdots+\ell_m=\ell}(\ell;\ell_1,\ldots,\ell_m)\prod_{j=1}^mz_j^{\ell_j}\left\{\f1{\Ga(\be_0+\bm\ell\cdot\bm\be)}+\sum_{j=1}^mz_j\,E_{\bm\be,\be_0+\bm\ell\cdot\bm\be+\be_j}(\bm z)\right\}\nonumber\\
& =\sum_{\ell_1+\cdots+\ell_m=\ell}\f{(\ell;\ell_1,\ldots,\ell_m)\prod_{j=1}^mz_j^{\ell_j}}{\Ga(\be_0+\bm\ell\cdot\bm\be)}+\sum_{j=1}^mz_j^{\ell+1}E_{\bm\be,\be_0+(\ell+1)\be_j}(\bm z)\nonumber\\
& \quad\,+\sum_{j=1}^m\sum_{\substack{\ell_1+\cdots+\ell_m=\ell\\\ell_j<\ell}}(\ell;\ell_1,\ldots,\ell_m)\,z_j\prod_{i=1}^mz_i^{\ell_i}E_{\bm\be,\be_0+\bm\ell\cdot\bm\be+\be_j}(\bm z)\label{eq-pf-1}\\
& =\sum_{\ell_1+\cdots+\ell_m=\ell}\f{(\ell;\ell_1,\ldots,\ell_m)\prod_{j=1}^mz_j^{\ell_j}}{\Ga(\be_0+\bm\ell\cdot\bm\be)}+\sum_{j=1}^mz_j^{\ell+1}E_{\bm\be,\be_0+(\ell+1)\be_j}(\bm z)\nonumber\\
& \quad\,+\sum_{j=1}^m\sum_{\substack{\ell_1+\cdots+\ell_m=\ell+1\\0<\ell_j<\ell+1}}(\ell;\ell_1,\ldots,\ell_{j-1},\ell_j-1,\ell_{j+1},\ldots,\ell_m)\prod_{i=1}^mz_i^{\ell_i}E_{\bm\be,\be_0+\bm\ell\cdot\bm\be}(\bm z)\label{eq-pf-2}\\
& =\sum_{\ell_1+\cdots+\ell_m=\ell}\f{(\ell;\ell_1,\ldots,\ell_m)\prod_{j=1}^mz_j^{\ell_j}}{\Ga(\be_0+\bm\ell\cdot\bm\be)}+\sum_{j=1}^mz_j^{\ell+1}E_{\bm\be,\be_0+(\ell+1)\be_j}(\bm z)\nonumber\\
& \quad\,+\sum_{\substack{\ell_1+\cdots+\ell_m=\ell+1\\\ell_j<\ell+1\ (\forall\,j)}}(\ell+1;\ell_1,\ldots,\ell_m)\prod_{i=1}^mz_i^{\ell_i}E_{\bm\be,\be_0+\bm\ell\cdot\bm\be}(\bm z)\label{eq-pf-3}\\
& =\sum_{\ell_1+\cdots+\ell_m=\ell}\f{(\ell;\ell_1,\ldots,\ell_m)\prod_{j=1}^mz_j^{\ell_j}}{\Ga(\be_0+\bm\ell\cdot\bm\be)}\nonumber\\
& \quad\,+\sum_{\ell_1+\cdots+\ell_m=\ell+1}(\ell+1;\ell_1,\ldots,\ell_m)\prod_{j=1}^mz_j^{\ell_j}E_{\bm\be,\be_0+\bm\ell\cdot\bm\be}(\bm z),\nonumber
\end{align}
where we distill the case of $\ell_j=\ell$ in the $j$th term in \eqref{eq-pf-1}, and substitute $\ell_j+1$ with $\ell_j$ in the other terms to obtain \eqref{eq-pf-2}. For \eqref{eq-pf-3}, we apply the multinomial coefficient formula (see, e.g., \cite[Equation (5.1)]{LLY15})
\[\sum_{j=1}^m(\ell;\ell_1,\ldots,\ell_{j-1},\ell_j-1,\ell_{j+1},\ldots,\ell_m)=(\ell+1;\ell_1,\ldots,\ell_m)\quad\mbox{with }\ell+1=\sum_{j=1}^m\ell_j.\]
Therefore, \eqref{eq-remainder} also holds true for $\ell+1$ and the proof is completed.
\end{proof}

The above fact, together with the boundedness result in Lemma \ref{lem-ML}(b), will play a crucial role in the proof of the strong maximum principle.

Concerning some important existing results of the solution to \eqref{eq-ibvp-u}, we shall prepare several lemmata for later use. First we state the well-posedness and the long-time asymptotic behavior of the solution to \eqref{eq-ibvp-u}.

\begin{lem}[see \cite{LLY15}]\label{lem-lly}
Fix $T>0$ arbitrarily. Concerning the solution $u$ to \eqref{eq-ibvp-u}, the followings hold true.

{\rm(a)} Let $a\in L^2(\Om)$ and $F=0$. Then there exists a unique solution $u\in C([0,T];L^2(\Om))\cap C((0,T];H^2(\Om)\cap H_0^1(\Om))$, which is explicitly written as
\begin{equation}\label{eq-rep-v}
u(\,\cdot\,,t)=\sum_{n=1}^\infty\left(1-\la_nt^{\al_1}E_{\bm\al',1+\al_1}^{(n)}(t)\right)(a,\vp_n)\,\vp_n
\end{equation}
in $C([0,T];L^2(\Om))\cap C((0,T];H^2(\Om)\cap H_0^1(\Om))$, where $\{(\la_n,\vp_n)\}_{n=1}^\infty$ is the eigensystem of $\cA$ and
\[E_{\bm\al',1+\al_1}^{(n)}(t):=E_{(\al_1,\al_1-\al_2,\ldots,\al_1-\al_m),1+\al_1}(-\la_nt^{\al_1},-q_2t^{\al_1-\al_2},\ldots,-q_mt^{\al_1-\al_m}).\]
Moreover, there exists a constant $C>0$ independent of $a$ such that
\begin{equation}\label{eq-est-H2}
\|u(\,\cdot\,,t)\|_{L^2(\Om)}\le C\|a\|_{L^2(\Om)},\quad\|u(\,\cdot\,,t)\|_{H^2(\Om)}\le C\,t^{-\al_1}\|a\|_{L^2(\Om)}\quad(0<t\le T).
\end{equation}
In addition, there holds
\begin{equation}\label{eq-asymp}
\left\|u(\,\cdot\,,t)-\f{q_m}{\Ga(1-\al_m)}\f{\cA^{-1}a}{t^{\al_m}}\right\|_{H^2(\Om)}\le\f{C\|a\|_{L^2(\Om)}}{t^{\min\{\al_{m-1},2\al_m\}}}\quad\mbox{as }t\to\infty.
\end{equation}

{\rm(b)} Let $a=0$ and $F\in L^\infty(0,T;L^2(\Om))$. Then there exists a unique solution $u\in L^2(0,T;$ $H^2(\Om)\cap H_0^1(\Om))$ such that $\lim_{t\to0}\|u(\,\cdot\,,t)\|_{L^2(\Om)}=0$.
\end{lem}

Regarding the time-analyticity of the solution to \eqref{eq-ibvp-u}, we have the following result.

\begin{lem}\label{lem-analy}
Let $a\in L^2(\Om)$, $F=0$ and $u$ be the solution to \eqref{eq-ibvp-u}. Then for arbitrarily small $\ve>0$, $u:(0,T]\to\cD(\cA^{1-\ve})$ can be analytically extended to a sector $\{z\in\BC;\,z\ne0,\ |\arg\,z|<\pi/2\}$.
\end{lem}

We note that the above statement is stronger than \cite[Theorem 4.1]{LY13a} which asserts the analyticity up to $H^1_0(\Om)$. Indeed, one can improve the regularity to arbitrarily close to $H^2(\Om)\cap H_0^1(\Om)$ by a more delicate reasoning, but here we skip the details.

Finally, we recall the weak maximum principle for \eqref{eq-ibvp-u}, which is the starting point of this paper.

\begin{lem}\label{lem-wmp}
Let $a\in L^2(\Om)$ and $F\in L^\infty(0,\infty;L^2(\Om))$ be nonnegative, and $u$ be the solution to \eqref{eq-ibvp-u}. Then there holds $u\ge0$ a.e.\! in $\Om\times(0,\infty)$.
\end{lem}

Similarly to that explained in \cite{LY13a}, the above conclusion is a special case of \cite[Theorem 3]{L11} but the settings here are more general. Since the same argument still works in our case, again we omit the proof here.

\Section{Proof of Theorem \ref{thm-smp} and Corollary \ref{coro-smp}}\label{sec-proof-smp}

Now we have collected all the necessities to investigate the strong maximum principle. Some parts of this sections are basically parallel to that in \cite{LRY15}, but essential difficulties occur with the appearance of the multinomial Mittag-Leffler functions.

In view of the weak maximum principle and the superposition principle, first we concentrate on the homogeneous problem, that is,
\begin{equation}\label{eq-ibvp-v}
\left\{\!\begin{alignedat}{2}
& \sum_{j=1}^mq_j\pa_t^{\al_j}v+\cA v=0 & \quad & \mbox{in }\Om\times(0,\infty),\\
& v=a & \quad & \mbox{in }\Om\times\{0\},\\
& v=0 & \quad & \mbox{on }\pa\Om\times(0,\infty).
\end{alignedat}\right.
\end{equation}
Henceforth we denote the solution to \eqref{eq-ibvp-v} as $v_a$ to emphasize its dependency upon the initial value $a$.

To begin with, we first show the strict positivity of the solution to \eqref{eq-ibvp-v} with a non-negative and non-vanishing initial value after sufficiently long time.

\begin{lem}\label{lem-asymp}
Let $v_a$ be the solution to \eqref{eq-ibvp-v} with $a\in L^2(\Om)$ satisfying $a\ge0$ and $a\not\equiv0$. Then for any $x\in\Om$, there exists $T>0$ sufficiently large such that
\begin{equation}\label{eq-long-pos}
v_a(x,t)>0\quad(\forall\,t\ge T).
\end{equation}
\end{lem}

\begin{proof}
According to the asymptotic behavior \eqref{eq-asymp} and the Sobolev embedding $H^2(\Om)\to C(\ov\Om)$ for $d\le3$, we obtain
\[\left|v_a(x,t)-\f{q_m}{\Ga(1-\al_m)}\f{b(x)}{t^{\al_m}}\right|\le C\left\|v_a(\,\cdot\,,t)-\f{q_m}{\Ga(1-\al_m)}\f b{t^{\al_m}}\right\|_{H^2(\Om)}\le\f{C\|a\|_{L^2(\Om)}}{t^{\min\{\al_{m-1},2\al_m\}}}\]
and thus
\begin{equation}\label{eq-est-b}
v_a(x,t)\ge\f{q_m}{\Ga(1-\al_m)}\f{b(x)}{t^{\al_m}}-\f{C\|a\|_{L^2(\Om)}}{t^{\min\{\al_{m-1},2\al_m\}}}
\end{equation}
for any $x\in\Om$ as $t\to\infty$, where $b:=\cA^{-1}a$, i.e., $b$ solves the boundary value problem
\begin{equation}\label{eq-bvp}
\begin{cases}
\cA b=a & \mbox{in }\Om,\\
b=0 & \mbox{on }\pa\Om.
\end{cases}
\end{equation}
Since $a\ge0$, $a\not\equiv0$ and the coefficient $c$ in the elliptic operator $\cA$ is non-negative (see \eqref{eq-def-A}), the strong maximum principle (see Gilbarg and Trudinger \cite[Chapter 3]{GT01}) for the elliptic equation \eqref{eq-bvp} indicates $b>0$ in $\Om$. Now that $b(x)>0$ and $\al_m<\min\{\al_{m-1},2\al_m\}$, there exists a sufficiently large $T>0$ such that the right-hand side of \eqref{eq-est-b} keeps strictly positive for all $t\ge T$, implying \eqref{eq-long-pos} immediately.
\end{proof}

Next we study the Green function of problem \eqref{eq-ibvp-v}. Using the multinomial Mittag-Leffler function and the eigensystem $\{(\la_n,\vp_n)\}_{n=1}^\infty$, for $N\in\BN$ we set
\[G_N(x,y,t):=\sum_{n=1}^N\left(1-\la_nt^{\al_1}E_{\bm\al',1+\al_1}^{(n)}(t)\right)\vp_n(x)\,\vp_n(y)\quad(x,y\in\Om,\ t>0).\]
Then it follows from Lemma \ref{lem-lly}(a) that
\[v_a(x,t)=\lim_{N\to\infty}\int_\Om G_N(x,y,t)\,a(y)\,\rd y\]
in $C([0,\infty);L^2(\Om))\cap C((0,\infty);H^2(\Om)\cap H_0^1(\Om))$ for any $a\in L^2(\Om)$ and any $t>0$. Therefore, $v_a(x,t)$ allows a pointwise definition and thus $G_N(x,\,\cdot\,,t)$ converges weakly to
\begin{equation}\label{eq-def-G}
G(x,y,t):=\sum_{n=1}^\infty\left(1-\la_nt^{\al_1}E_{\bm\al',1+\al_1}^{(n)}(t)\right)\vp_n(x)\,\vp_n(y)
\end{equation}
as a series of functions with respect to $y$. In particular, we obtain $G(x,\,\cdot\,,t)\in L^2(\Om)$ for all $x\in\Om$ and all $t>0$. Moreover, we can rewrite the solution to \eqref{eq-ibvp-v} as
\begin{equation}\label{eq-rep-G}
v_a(x,t)=\int_\Om G(x,y,t)\,a(y)\,\rd y\quad(x\in\Om,\ t>0).
\end{equation}
On the other hand, the same reasoning as that in \cite{LRY15} and Lemma \ref{lem-wmp} indicate $G(x,\,\cdot\,,t)\ge0$ a.e.\! in $\Om$ for arbitrarily fixed $x\in\Om$ and $t>0$, and we summarize the above observations as follows.

\begin{lem}\label{lem-Green}
Let $G(x,y,t)$ be the Green function defined in \eqref{eq-def-G}. Then for arbitrarily fixed $x\in\Om$ and $t>0$, we have
\[G(x,\,\cdot\,,t)\in L^2(\Om)\quad\mbox{and}\quad G(x,\,\cdot\,,t)\ge0\ \mbox{a.e.\! in }\Om.\]
\end{lem}

Now we are well prepared to prove the strong maximum principle.

\begin{proof}[Proof of Theorem $\ref{thm-smp}$]
(a) We first deal with the homogeneous problem \eqref{eq-ibvp-v}, i.e., $F=0$. Fix an initial value $a\in L^2(\Om)$ such that $a\ge0$ and $a\not\equiv0$. By Lemma \ref{lem-lly}(a) and the Sobolev embedding $H^2(\Om)\subset C(\ov\Om)$ for $d\le3$, we have $u\in C(\ov\Om\times(0,\infty))$. Meanwhile, the weak maximum principle stated in Lemma \ref{lem-wmp} indicates $v_a\ge0$ in $\Om\times(0,\infty)$ and hence
\[\cE_x:=\{t>0;\,v_a(x,t)\le0\}=\{t>0;\,v_a(x,t)=0\},\]
namely, $\cE_x$ is the zero point set of $v_a(x,t)$ as a non-negative function of $t>0$. Further, since Lemma \ref{lem-asymp} guarantees a sufficiently large $T>0$ such that \eqref{eq-long-pos} holds, we conclude that the possible elements in $\cE_x$ distribute in the finite interval $(0,T)$.

Assume contrarily that there exists $x_0\in\Om$ such that $\cE_{x_0}$ is not a finite set. Then $\cE_{x_0}$ contains at least an accumulation point. We first deal with the case of a positive accumulation point. According to Lemma \ref{lem-analy}, we have the analyticity of $v_a:(0,\infty)\to\cD(\cA^{1-\ve})\subset C(\ov\Om)$ for arbitrarily small $\ve>0$, indicating that $v_a(x_0,t)$ is analytic with respect to $t>0$. Therefore, $v_a(x_0,t)$ should vanish identically if its zero points accumulate at some $t>0$, which contradicts with Lemma \ref{lem-asymp}.

If the zero points of $v_a(x_0,t)$ accumulate at $t=0$, we should argue separately since $v_a(x_0,t)$ is not analytic at $t=0$. Henceforth $C>0$ denotes generic constants independent of $n\in\BN$ and $t\ge0$, which may change line by line.

By assumption, there is a sequence $\{t_i\}_{i=1}^\infty\subset\cE_{x_0}$ such that $t_i\to0$ ($i\to\infty$) and, by the representation \eqref{eq-rep-G},
\[v_a(x_0,t_i)=\int_\Om G(x_0,y,t_i)\,a(y)\,\rd y=0\quad(i=1,2,\ldots).\]
Since $a\ge0$ and $G(x_0,\,\cdot\,,t_i)\ge0$ by Lemma \ref{lem-Green}, we deduce $G(x_0,y,t_i)\,a(y)=0$ for all $i=1,2,\ldots$ and almost all $y\in\Om$. Moreover, the non-vanishing assumption on $a$ implies that $G(x_0,\,\cdot\,,t_i)$ should vanish in the subdomain $\om:=\{a>0\}$ whose measure is positive. In view of the representation \eqref{eq-def-G}, this indicates
\begin{equation}\label{eq-van-1}
\sum_{n=1}^\infty\left(1-\la_nt^{\al_1}E_{\bm\al',1+\al_1}^{(n)}(t_i)\right)\vp_n(x_0)\,\vp_n=0\quad\mbox{a.e.\! in }\om\ (i=1,2,\ldots).
\end{equation}

Now we choose $\psi\in C^\infty_0(\om)$ arbitrarily as the initial data of \eqref{eq-ibvp-v} and investigate
\begin{equation}\label{eq-van-def}
v_\psi(x_0,t)=\int_\Om G(x_0,y,t)\,\psi(y)\,\rd y=\sum_{n=1}^\infty\left(1-\la_nt^{\al_1}E_{\bm\al',1+\al_1}^{(n)}(t)\right)\psi_n,
\end{equation}
where we abbreviate $\psi_n:=(\psi,\vp_n)\,\vp_n(x_0)$. By the same estimates as that in \cite{LRY15}, we employ Lemma \ref{lem-ML}(b) to conclude that the series in \eqref{eq-van-def} is well-defined in $C[0,T]$. Moreover, repeating the same argument, it is not difficult to find
\begin{equation}\label{eq-cov}
\sum_{n=1}^\infty\left|\la_n^\ell\,E_{\bm\al',\be}^{(n)}(t)\,\psi_n\right|<\infty\quad(\forall\,\ell=0,1,\ldots,\ \forall\,\be>0,\ \forall\,t\in[0,T],\ \forall\,\psi\in C_0^\infty(\om)).
\end{equation}

Now we can substitute \eqref{eq-van-1} into \eqref{eq-van-def} to deduce
\[v_\psi(x_0,t_i)=\sum_{n=1}^\infty\left(1-\la_nt_i^{\al_1}E_{\bm\al',1+\al_1}^{(n)}(t_i)\right)\psi_n=0\quad(i=1,2,\ldots,\ \forall\,\psi\in C_0^\infty(\om)).\]
Passing $i\to\infty$ and using the boundedness of $\sum_{n=1}^\infty\la_n\,E_{\bm\al',1+\al_1}^{(n)}(t_i)\,\psi_n$, we obtain
\begin{equation}\label{eq-van-2}
\psi(x_0)=\sum_{n=1}^\infty(\psi,\vp_n)\,\vp_n(x_0)=\sum_{n=1}^\infty\psi_n=0,
\end{equation}
implying
\begin{equation}\label{eq-ind-0}
v_\psi(x_0,t)=-t^{\al_1}\sum_{n=1}^\infty\la_n\,E_{\bm\al',1+\al_1}^{(n)}(t)\,\psi_n.
\end{equation}
Meanwhile, since $\psi\in C_0^\infty(\om)$ is arbitrarily chosen, \eqref{eq-van-2} indicates that the only possibility is $x_0\notin\om$. Therefore, there holds for all $\ell=0,1,\ldots$ that
\begin{equation}\label{eq-van-3}
\sum_{n=1}^\infty\la_n^\ell\,\psi_n=\sum_{n=1}^\infty(\psi,\cA^\ell\vp_n)\,\vp_n(x_0)=\sum_{n=1}^\infty(\cA^\ell\psi,\vp_n)\,\vp_n(x_0)=\cA^\ell\psi(x_0)=0.
\end{equation}

In the next step, we continue the treatment for \eqref{eq-ind-0}. Our aim is to demonstrate
\begin{equation}\label{eq-van-4}
v_\psi(x_0,t)=0\quad(0\le t\le T,\ \forall\,\psi\in C_0^\infty(\om)).
\end{equation}
To this end, recall the notation $R_{\bm\be,\be_0}^{\,\ell}(\bm z)$ for the remainder after the $\ell$th term in the series by which the multinomial Mittag-Leffler function is defined (see Lemma \ref{lem-remainder}). Similarly, here we denote the remainder of $E_{\bm\al',1+\al_1}^{(n)}(t)$ by $R_{\bm\al',1+\al_1}^{(n),\ell}(t)$. We claim that
\begin{equation}\label{eq-ind-l}
v_\psi(x_0,t)=-t^{\al_1}\sum_{n=1}^\infty\la_n\,R_{\bm\al',1+\al_1}^{(n),\ell}(t)\,\psi_n\quad(\forall\,\ell=1,2,\ldots)
\end{equation}
in the sense of $C[0,T]$. To this end, we proceed by induction. In the case of $\ell=1$, it follows from \eqref{eq-van-3} and Lemma \ref{lem-remainder} with $\ell=1$ that
\begin{align*}
v_\psi(x_0,t) & =-\f{t^{\al_1}}{\Ga(1+\al_1)}\sum_{n=1}^\infty\la_n\psi_n-t^{\al_1}\sum_{n=1}^\infty\la_n\,R_{\bm\al',1+\al_1}^{(n),1}(t)\,\psi_n=-t^{\al_1}\sum_{n=1}^\infty\la_n\,R_{\bm\al',1+\al_1}^{(n),1}(t)\,\psi_n\\
& =t^{2\al_1}\sum_{n=1}^\infty\la_n^2\,E_{\bm\al',1+2\al_1}^{(n)}(t)\,\psi_n+\sum_{j=2}^mq_jt^{2\al_1-\al_j}\sum_{n=1}^\infty\la_n\,E_{\bm\al',1+2\al_1-\al_j}^{(n)}(t)\,\psi_n,
\end{align*}
where the involved summations with respect to $n$ are all absolutely convergent in $[0,T]$ by \eqref{eq-cov}. Now suppose that \eqref{eq-ind-l} is valid for some $\ell$. Then we directly calculate
\[R_{\bm\al',1+\al_1}^{(n),\ell}=(-1)^\ell\sum_{\ell_1+\cdots+\ell_m=\ell}\f{(\ell;\ell_1,\ldots,\ell_m)\,\la_n^{\ell_1}t^{\al_1\ell}\prod_{j=2}^m(q_jt^{-\al_j})^{\ell_j}}{\Ga(1+\al_1+\bm\ell\cdot\bm\al')}+R_{\bm\al',1+\al_1}^{(n),\ell+1}(t).\]
Substituting the above equality into \eqref{eq-ind-l}, we again employ \eqref{eq-van-3} and Lemma \ref{lem-remainder} to deduce
\begin{align*}
v_\psi(x_0,t) & =(-1)^\ell\sum_{\ell_1+\cdots+\ell_m=\ell}\f{(\ell;\ell_1,\ldots,\ell_m)\,t^{\al_1(\ell+1)}\prod_{j=2}^m(q_jt^{-\al_j})^{\ell_j}}{\Ga(1+\al_1+\bm\ell\cdot\bm\al')}\sum_{n=1}^\infty\la_n^{\ell_1+1}\psi_n\\
& \quad\,-t^{\al_1}\sum_{n=1}^\infty\la_n\,R_{\bm\al',1+\al_1}^{(n),\ell+1}(t)\,\psi_n\\
& =-t^{\al_1}\sum_{n=1}^\infty\la_n\,R_{\bm\al',1+\al_1}^{(n),\ell+1}(t)\,\psi_n\\
& =(-1)^\ell\sum_{\ell_1+\cdots+\ell_m=\ell+1}(\ell+1;\ell_1,\ldots,\ell_m)\,t^{\al_1(\ell+1)}\prod_{j=2}^m(q_jt^{-\al_j})^{\ell_j}\sum_{n=1}^\infty\la_n^{\ell_1+1}E_{\bm\al',1+\al_1}^{(n)}(t)\,\psi_n,
\end{align*}
where, again, all of the involved summations with respect to $n$ are absolutely convergent in $[0,T]$ by \eqref{eq-cov}. Therefore, \eqref{eq-ind-l} also holds true for $\ell+1$ and thus holds for all $\ell=1,2,\ldots$.

To conclude \eqref{eq-van-4}, now it suffices to show
\begin{equation}\label{eq-van-5}
\lim_{\ell\to\infty}\sum_{n=1}^\infty\la_n\,R_{\bm\al',1+\al_1}^{(n),\ell}(t)\,\psi_n=0\quad(0\le t\le T).
\end{equation}
Actually, we note the fact that $R_{\bm\al',1+\al_1}^{(n),\ell}(t)$ stands for the remainder of the series defining the multinomial Mittag-Leffler function $E_{\bm\al',1+\al_1}^{(n)}(t)$, which converges uniformly in $[0,T]$ for all $n=1,2,\ldots$. In other words, we have
\[\lim_{\ell\to\infty}R_{\bm\al',1+\al_1}^{(n),\ell}(t)=0\quad(\forall\,n=1,2,\ldots,\ 0\le t\le T),\]
which, together with the boundedness of $\sum_{n=1}^\infty|\la_n\psi_n|$, yields \eqref{eq-van-5} immediately. Since $v_\psi(x_0,t)$ is independent of $\ell$, we apply \eqref{eq-van-5} to \eqref{eq-ind-l} to obtain
\[v_\psi(x_0,t)=-t^{\al_1}\lim_{\ell\to\infty}\sum_{n=1}^\infty\la_n\,R_{\bm\al',1+\al_1}^{(n),\ell}(t)\,\psi_n=0\quad(0\le t\le T,\ \forall\,\psi\in C_0^\infty(\om)),\]
that is, \eqref{eq-van-4}. Note that the choice of $T>0$ is arbitrary.

As the final step, we specify $\psi_0\in C_0^\infty(\om)$ such that $\psi_0\ge0$ and $\psi_0\not\equiv0$. Then the application of Lemma \ref{lem-asymp} guarantees a sufficiently large constant $T_0>0$ such that $v_{\psi_0}(x_0,t)>0$ for all $t\ge T_0$. However, taking $\psi=\psi_0$ and $T=T_0$ in \eqref{eq-van-4}, we are led to $v_{\psi_0}(x_0,T_0)=0$, which is a contradiction. Consequently, we have proved that $t=0$ cannot be an accumulation point of $\cE_{x_0}$.

Therefore, for any $x\in\Om$, we have excluded all the possibilities for $\cE_x$ to possess any accumulation point in $[0,\infty]$, indicating that $\cE_x$ is a finite set. To further conclude $u>0$ a.e.\! in $\Om\times(0,\infty)$, it suffices to show that $D:=\{(x,t)\in\Om\times(0,\infty);\,u(x,t)\le0\}$ is a set of zero measure. Since $D\cap(\{x\}\times(0,\infty))=\cE_x$ and the characterization function $\chi_{\cE_x}=0$ a.e.\! in $(0,\infty)$, it follows immediately from Fubini's theorem that
\[\mathrm{meas}(D)=\int_{\Om\times(0,\infty)}\chi_D(x,t)\,\rd x\rd t=\int_\Om\int_0^\infty\chi_{\cE_x}(t)\,\rd t\rd x=0.\]

(b) Now we turn to the inhomogeneous problem with a non-negative source term $F$. Due to the linearity of the problem with respect to $a$ and $F$, we know $u=v_a+w$, where $v_a$ and $w$ solve \eqref{eq-ibvp-v} and
\[\left\{\!\begin{alignedat}{2}
& \sum_{j=1}^mq_j\pa_t^{\al_j}w+\cA w=F & \quad & \mbox{in }\Om\times(0,T],\\
& w=0 & \quad & \mbox{in }\Om\times\{0\},\\
& w=0 & \quad & \mbox{on }\pa\Om\times(0,T],
\end{alignedat}\right.\]
respectively. Then it follows immediately from (a) that $v_a>0$ a.e.\! in $\Om\times(0,\infty)$. On the other hand, since $F\ge0$, the weak maximum principle (see Lemma \ref{lem-wmp}) implies $w\ge0$ a.e.\! in $\Om\times(0,\infty)$. Hence, it is readily seen that $u=v_a+w>0$ a.e.\! in $\Om\times(0,\infty)$.
\end{proof}

\begin{proof}[Proof of Corollary $\ref{coro-smp}$]
The argument basically follows the same line as that for \cite[Corollary 1.1]{LRY15}. Similarly to the proof of Theorem \ref{thm-smp}(b), it suffices to investigate the homogeneous problem \eqref{eq-ibvp-v} whose initial data $a\in L^2(\Om)$ is strictly positive. Recall that in this case we have $v_a\in C(\ov\Om\times(0,\infty))$, and its non-negativity is guaranteed by Lemma \ref{lem-wmp}.

Assume on the contrary that there exists a pair $(x_0,t_0)\in\Om\times(0,\infty)$ such that $v_a(x_0,t_0)=0$. Taking advantage of the representation \eqref{eq-rep-G}, we have $\int_\Om G(x_0,y,t_0)\,a(y)\,\rd y=0$. Since $G(x_0,\,\cdot\,,t_0)\ge0$ by Lemma \ref{lem-Green} and $a>0$ by assumption, there should hold $G(x_0,\,\cdot\,,t_0)=0$ in $\Om$, that is,
\[\sum_{n=1}^\infty\left(1-\la_nt_0^{\al_1}E_{\bm\al',1+\al_1}^{(n)}(t_0)\right)\vp_n(x_0)\,\vp_n=0\quad\mbox{in }\Om.\]
Then the complete orthogonality of $\{\vp_n\}$ in $L^2(\Om)$ immediately yields
\[\left(1-\la_nt_0^{\al_1}E_{\bm\al',1+\al_1}^{(n)}(t_0)\right)\vp_n(x_0)=0,\quad\forall\,n=1,2,\ldots,\]
especially, $(1-\la_1t_0^{\al_1}E_{\bm\al',1+\al_1}^{(1)}(t_0))\,\vp_1(x_0)=0$. However, we know $\vp_1(x_0)>0$ because the first eigenfunction $\vp_1$ is strictly positive (see, e.g., Evans \cite{E10}). On the other hand, Bazhlekova \cite[Theorem 3.2]{B13} asserts that the function $1-\la_1t^{\al_1}E_{\bm\al',1+\al_1}^{(1)}(t)$ is completely monotone for $t\ge0$ and thus also keeps strictly positive, which results in a contradiction. Consequently, we see that such a pair $(x_0,t_0)$ does not exist, which completes the proof.
\end{proof}

\Section{Proof of Theorem \ref{thm-ISP}}\label{sec-proof-ISP}

Now we proceed to the proof of the uniqueness of the inverse source problem for the following initial-boundary value problem
\begin{equation}\label{eq-ibvp-w}
\left\{\!\begin{alignedat}{2}
& \sum_{j=1}^mq_j\pa_t^{\al_j}u(x,t)+\cA u(x,t)=\rho(t)\,g(x) & \quad & (x\in\Om,\ 0<t\le T),\\
& u(x,0)=0 & \quad & (x\in\Om),\\
& u(x,t)=0 & \quad & (x\in\pa\Om,\ 0<t\le T).
\end{alignedat}\right.
\end{equation}
Recall that the unknown temporal component $\rho$ is assumed to be in $C^1[0,T]$, and the known spatial component $g$ satisfies $g\in\cD(\cA^\ve)$ with some $\ve>0$, $g\ge0$ and $g\not\equiv0$.

Parallelly to the strategy in \cite[Section 4]{LRY15}, we shall first develop a fractional Duhamel's principle for \eqref{eq-ibvp-w} to relate the inhomogeneous problem with the homogeneous one, so that we can apply the established strong maximum principle. Nevertheless, to deal with the multi-term case, we shall invoke the Riemann-Liouville derivative
\begin{equation}\label{eq-def-RL}
D_t^\be h(t):=\f\rd{\rd t}J^{1-\be}h(t),\quad J^{1-\be}h(t):=\f1{\Ga(1-\be)}\int_0^t\f{h(s)}{(t-s)^\be}\,\rd s\quad(0<\be<1)
\end{equation}
and turn to the following lemma concerning the related fractional ordinary differential equation.

\begin{lem}\label{lem-RL}
Let $\al_j,q_j$ be the constants as that in \eqref{eq-ibvp-w} and $h_j\in\BR\ (j=1,2,\ldots,m)$ also be constants. Regarding the initial value problem
\begin{equation}\label{eq-ivp}
\left\{\!\begin{alignedat}{2}
& \sum_{j=1}^mq_jD_t^{\al_j}h(t)=f(t) & \quad & (0<t\le T<\infty),\\
& J^{1-\al_j}h(0)=h_j & \quad & (j=1,2,\ldots,m),
\end{alignedat}\right.
\end{equation}
the followings hold true.

{\rm(a)} If $f\in L^1(0,T)$, then \eqref{eq-ivp} has a unique solution $h\in L^1(0,T)$.

{\rm(b)} If $f\in C[0,T]$ and $h_j=0\ (j=1,2,\ldots,m)$, then \eqref{eq-ivp} has a unique solution $h\in C[0,T]$.

{\rm(c)} Let $\wt h(t)$ be the solution to
\[\left\{\!\begin{alignedat}{2}
& \sum_{j=1}^mq_jD_t^{\al_j}\wt h(t)=f(t) & \quad & (0<t\le T),\\
& J^{1-\al_j}\wt h(0)=\wt h_j & \quad & (j=1,2,\ldots,m).
\end{alignedat}\right.\]
Then the difference between the solution $h(t)$ of \eqref{eq-ivp} and $\wt h(t)$ is dominated as
\[|h(t)-\wt h(t)|\le\sum_{j=1}^m|h_j-\wt h_j|\,t^{\al_j-1}\quad(0<t\le T).\]
\end{lem}

The above lemma collects the corresponding results in \cite[Chapter 3]{P99} where slightly more general cases were treated. But this is sufficient for proving the following fractional Duhamel's principle for \eqref{eq-ibvp-w}.

\begin{lem}\label{lem-Duhamel}
Let $u$ be the solution to \eqref{eq-ibvp-w}, where $\rho\in C^1[0,T]$ and $g\in\cD(\cA^\ve)$ with some $\ve>0$. Then $u$ allows the representation
\[u(\,\cdot\,,t)=\int_0^t\mu(t-s)\,v_g(\,\cdot\,,s)\,\rd s\quad(0<t\le T),\]
where $v_g$ solves the homogeneous problem \eqref{eq-ibvp-v} with $g$ as the initial data, and $\mu$ satisfies
\begin{equation}\label{eq-def-mu}
\sum_{j=1}^mq_jJ^{1-\al_j}\mu(t)=\rho(t)\quad(0<t\le T)
\end{equation}
$($see \eqref{eq-def-RL} for the definition of $J^{1-\al_j})$. Moreover, there exists a unique $\mu\in L^1(0,T)$ satisfying \eqref{eq-def-mu}, and there is a constant $C>0$ independent of $t$ such that
\begin{equation}\label{eq-est-mu}
|\mu(t)|\le C\,t^{\al_1-1}\quad\mbox{a.e.\! }t\in(0,T).
\end{equation}
\end{lem}

\begin{proof}
Henceforth $C>0$ denotes generic constants independent of the spatial component $g$ and the time $t$. To investigate the unique existence of the solution to \eqref{eq-def-mu} with $\rho\in C^1[0,T]$, we differentiate both sides of \eqref{eq-def-mu} to get an equivalent form
\begin{align}
& \sum_{j=1}^mq_jD_t^{\al_j}\mu(t)=\rho'(t)\quad(0<t\le T),\label{eq-ode-RL}\\
& \sum_{j=1}^mq_jJ^{1-\al_j}\mu(0)=\rho(0).\nonumber
\end{align}
By simple analysis of the singularity at $t=0$, it reveals that the initial condition above can be reduced to
\begin{equation}\label{eq-ode-IC}
J^{1-\al_1}\mu(0)=\rho(0),\quad J^{1-\al_j}\mu(0)=0\ (j=2,\ldots,m).
\end{equation}
Then \eqref{eq-ode-RL}--\eqref{eq-ode-IC} becomes an initial value problem for an ordinary differential equation with multiple Riemann-Liouville derivatives. Further, since $\rho'\in C[0,T]\subset L^1(0,T)$, it immediately follows from Lemma \ref{lem-RL}(a) that there exists a unique solution $\mu\in L^1(0,T)$.

In order to show the estimate \eqref{eq-est-mu}, we introduce the auxiliary problem
\[\left\{\!\begin{alignedat}{2}
& \sum_{j=1}^mq_jD_t^{\al_j}\wt\mu(t)=\rho'(t) & \quad & (0<t\le T),\\
& J^{1-\al_j}\wt\mu(t)=0 & \quad & (j=1,\ldots,m),
\end{alignedat}\right.\]
that is, equation \eqref{eq-ode-RL} with the homogeneous initial condition. By the continuity of $\rho'$, we have $\wt\mu\in C[0,T]$ according to Lemma \ref{lem-RL}(b). Moreover, Lemma \ref{lem-RL}(c) indicates that the difference between $\mu$ and $\wt\mu$ is dominated by $C\,t^{\al_1-1}$ a.e.\! in $(0,T)$. In summary, we obtain
\[|\mu(t)|\le|\mu(t)-\wt\mu(t)|+|\wt\mu(t)|\le C\,t^{\al_1-1}+C\le C\,t^{\al_1-1}\quad\mbox{a.e.\! }t\in(0,T).\]

Now we consider the initial-boundary value problems \eqref{eq-ibvp-w} and \eqref{eq-ibvp-v} for $u$ and $v_g$. Since $\rho\,g\in L^\infty(0,T;L^2(\Om))$, Lemma \ref{lem-lly}(b) gives
\[u\in L^2(0,T;H^2(\Om)\cap H_0^1(\Om)),\quad\lim_{t\to0}\|u(\,\cdot\,,t)\|_{L^2(\Om)}=0.\]
By setting
\begin{equation}\label{eq-def-wtu}
\wt u(\,\cdot\,,t):=\int_0^t\mu(t-s)\,v_g(\,\cdot\,,s)\,\rd s,
\end{equation}
we shall demonstrate
\[u=\wt u\mbox{ in }L^2(0,T;H^2(\Om)\cap H_0^1(\Om)),\quad\lim_{t\to0}\|\wt u(\,\cdot\,,t)\|_{L^2(\Om)}=0.\]
To this end, we first investigate $v_g$. Since $g\in\cD(\cA^\ve)\subset L^2(\Om)$, the application of estimate \eqref{eq-est-H2} in Lemma \ref{lem-lly}(a) yields
\[\|v_g(\,\cdot\,,t)\|_{L^2(\Om)}\le C\|g\|_{L^2(\Om)},\quad\|v_g(\,\cdot\,,t)\|_{H^2(\Om)}\le C\|g\|_{L^2(\Om)}\,t^{-\al_1}\quad(0<t\le T).\]
Then we utilize \eqref{eq-est-mu} and \eqref{eq-def-wtu} to estimate
\begin{align*}
\|\wt u(\,\cdot\,,t)\|_{L^2(\Om)} & \le\int_0^t|\mu(t-s)|\|v_g(\,\cdot\,,s)\|_{L^2(\Om)}\,\rd s\le C\|g\|_{L^2(\Om)}\int_0^ts^{\al_1-1}\,\rd s\\
& \le C\|g\|_{L^2(\Om)}\,t^{\al_1}\to0\quad(t\to0),\\
\|\wt u(\,\cdot\,,t)\|_{H^2(\Om)} & \le\int_0^t|\mu(t-s)|\|v_g(\,\cdot\,,s)\|_{H^2(\Om)}\,\rd s\le C\|g\|_{L^2(\Om)}\int_0^t(t-s)^{\al_1-1}s^{-\al_1}\,\rd s\\
& \le C\|g\|_{L^2(\Om)}\quad(0<t\le T),
\end{align*}
which indicates $\wt u\in L^\infty(0,T;H^2(\Om)\cap H_0^1(\Om))\subset L^2(0,T;H^2(\Om)\cap H_0^1(\Om))$. On the other hand, according to the explicit representation \eqref{eq-rep-v} and the identity (see \cite[Lemma 3.3]{LLY15})
\[\f\rd{\rd t}\left\{t^{\al_1}E_{\bm\al',1+\al_1}^{(n)}(t)\right\}=t^{\al_1-1}E_{\bm\al',\al_1}^{(n)}(t),\]
we obtain
\[\pa_tv(\,\cdot\,,t)=-t^{\al_1-1}\sum_{n=1}^\infty\la_n\,E_{\bm\al',\al_1}^{(n)}(t)\,(g,\vp_n)\,\vp_n.\]
By Lemma \ref{lem-ML}(b) and the fact $g\in\cD(\cA^\ve)$ with $\ve>0$, we estimate for $0<t\le T$ that
\begin{align}
\|\pa_tv_g(\,\cdot\,,t)\|_{L^2(\Om)}^2 & =t^{2(\al_1-1)}\sum_{n=1}^\infty\left|\la_n\,E_{\bm\al',\al_1}^{(n)}(t)\,(g,\vp_n)\right|^2\nonumber\\
& =t^{2(\al_1-1)}\sum_{n=1}^\infty\left|\la_n^{1-\ve}E_{\bm\al',\al_1}^{(n)}(t)\right|^2|\la_n^\ve\,(g,\vp_n)|^2\nonumber\\
& \le\left(C\,t^{\al_1\ve-1}\right)^2\sum_{n=1}^\infty\left|\f{(\la_nt^{\al_1})^{1-\ve}}{1+\la_nt^{\al_1}}\right|^2|\la_n^\ve\,(g,\vp_n)|^2\le\left(C\|g\|_{\cD(\cA^\ve)}t^{\al_1\ve-1}\right)^2.\label{eq-est-vt}
\end{align}

To show $u=\wt u$, now it suffices to verify that $\wt u$ also solves the initial-boundary value problem \eqref{eq-ibvp-w} which has a unique solution (see Lemma \ref{lem-lly}(b)). To calculate $\pa_t^\al\wt u$, first we formally calculate
\begin{equation}\label{eq-wtut}
\pa_t\wt u(\,\cdot\,,t)=\pa_t\int_0^t\mu(s)\,v_g(\,\cdot\,,t-s)\,\rd s=\int_0^t\mu(s)\,\pa_tv_g(\,\cdot\,,t-s)\,\rd s+\mu(t)\,g.
\end{equation}
Then we employ \eqref{eq-est-mu} and \eqref{eq-est-vt} to estimate
\begin{align*}
\|\pa_t\wt u(\,\cdot\,,t)\|_{L^2(\Om)} & \le\int_0^t|\mu(t-s)|\|\pa_sv_g(\,\cdot\,,s)\|_{L^2(\Om)}\,\rd s+|\mu(t)|\|g\|_{L^2(\Om)}\\
& \le C\|g\|_{\cD(\cA^\ve)}\int_0^t(t-s)^{\al_1-1}s^{\al_1\ve-1}\,\rd s+C\|g\|_{L^2(\Om)}\,t^{\al_1-1}\\
& \le C\|g\|_{\cD(\cA^\ve)}\,t^{\al_1(1+\ve)-1}+C\|g\|_{L^2(\Om)}\,t^{\al_1-1}\le C\|g\|_{\cD(\cA^\ve)}\,t^{\al_1-1}\quad(0<t\le T),
\end{align*}
implying that the above differentiation makes sense in $L^2(\Om)$ for $0<t\le T$. Finally, we calculate by the definition of Caputo derivatives, \eqref{eq-wtut} and \eqref{eq-def-mu} to conclude
\begin{align*}
\sum_{j=1}^mq_j\pa_t^{\al_j}\wt u(\,\cdot\,,t) & =\sum_{j=1}^m\f{q_j}{\Ga(1-\al_j)}\int_0^t\f{\pa_s\wt u(\,\cdot\,,s)}{(t-s)^{\al_j}}\,\rd s\\
& =\sum_{j=1}^m\f{q_j}{\Ga(1-\al_j)}\int_0^t\f1{(t-s)^{\al_j}}\int_0^s\mu(\tau)\,\pa_sv_g(\,\cdot\,,s-\tau)\,\rd\tau\rd s\\
& \quad\,+g\sum_{j=1}^m\f{q_j}{\Ga(1-\al_j)}\int_0^t\f{\mu(s)}{(t-s)^{\al_j}}\,\rd s\\
& =\sum_{j=1}^m\f{q_j}{\Ga(1-\al_j)}\int_0^t\mu(\tau)\int_\tau^t\f{\pa_sv_g(\,\cdot\,,s-\tau)}{(t-s)^{\al_j}}\,\rd s\rd\tau+g\sum_{j=1}^mq_jJ^{1-\al_j}\mu(t)\\
& =\sum_{j=1}^mq_j\int_0^t\f{\mu(\tau)}{\Ga(1-\al_j)}\int_0^{t-\tau}\f{\pa_sv_g(\,\cdot\,,s)}{((t-\tau)-s)^{\al_j}}\,\rd s\rd\tau+\rho(t)\,g\\
& =\int_0^t\mu(\tau)\sum_{j=1}^mq_j\pa_t^{\al_j}v_g(\,\cdot\,,t-\tau)\,\rd\tau+\rho(t)\,g\\
& =-\int_0^t\mu(\tau)\,\cA v_g(\,\cdot\,,t-\tau)\,\rd\tau+\rho(t)\,g=-\cA\int_0^t\mu(\tau)\,v_g(\,\cdot\,,t-\tau)\,\rd\tau+\rho(t)\,g\\
& =-\cA\wt u(\,\cdot\,,t)+\rho(t)\,g,
\end{align*}
that is, $\wt u$ also satisfies \eqref{eq-ibvp-w}. The proof of Lemma \ref{lem-Duhamel} is completed.
\end{proof}

At this stage, we are ready to show Theorem \ref{thm-ISP} by applying the above fractional Duhamel's principle and the strong maximum principle.

\begin{proof}[Completion of the Proof of Theorem $\ref{thm-ISP}$]
We follow the same argument as the proof of \cite[Theorem 1.2]{LRY15}. Recall the assumptions $\rho\in C^1[0,T]$, $g\in\cD(\cA^\ve)$ with some $\ve>0$, $g\ge0$, $g\not\equiv0$, and the solution $u$ to \eqref{eq-ibvp-w} vanishes in $\{x_0\}\times[0,T]$ for some $x_0\in\Om$. Taking advantage of Lemma \ref{lem-Duhamel}, we have
\begin{equation}\label{eq-Titchmarsh}
u(x_0,t)=\int_0^t\mu(t-s)\,v_g(x_0,s)\,\rd s=0\quad(0\le t\le T),
\end{equation}
where $\mu$ was defined in \eqref{eq-def-mu} and $v_g$ is the solution to \eqref{eq-ibvp-v} with the initial data $g$. Now we use the estimate \eqref{eq-est-H2} in Lemma \ref{lem-lly}(a) and the Sobolev embedding to deduce
\[|v_g(x_0,t)|\le C\|v_g(\,\cdot\,,t)\|_{H^2(\Om)}\le C\|g\|_{L^2(\Om)}\,t^{-\al_1}\quad(0<t\le T),\]
implying $v_g(x_0,\,\cdot\,)\in L^1(0,T)$. On the other hand, Lemma \ref{lem-Duhamel} gives $\mu\in L^1(0,T)$. Therefore, the application of the Titchmarsh convolution theorem (see \cite{T26}) to \eqref{eq-Titchmarsh} guarantees two constants $T_1,T_2\ge0$ satisfying $T_1+T_2\ge T$ such that $\mu=0$ a.e.\! in $(0,T_1)$ and $v_g(x_0,\,\cdot\,)=0$ a.e.\! in $(0,T_2)$. However, since the initial value $g$ is non-negative and non-vanishing, Theorem \ref{thm-smp} asserts that $v_g(x_0,\,\cdot\,)>0$ a.e.\! in $(0,T)$. As a result, the only possibility is $T_2=0$ and thus $T_1=T$, that is, $\mu=0$ a.e.\! in $(0,T)$. Finally, we apply Young's inequality to the relation \eqref{eq-def-mu} to conclude
\[\|\rho\|_{L^1(0,T)}\le\sum_{j=1}^mq_j\|J^{1-\al_j}\mu\|_{L^1(0,T)}\le C\sum_{j=1}^m\left\|\int_0^t\f{\mu(s)}{(t-s)^{\al_j}}\,\rd s\right\|_{L^1(0,T)}\le C\|\mu\|_{L^1(0,T)}=0,\]
which finishes the proof.
\end{proof}

\end{document}